\documentclass[intlim,righttag,10pt]{amsart}
\usepackage{amscd}
\usepackage{amssymb}
\usepackage[all]{xy}
\usepackage[french,british]{babel}
\usepackage{hyperref}
\usepackage{url}
\usepackage{xcolor}
\usepackage{dutchcal}
\newtheorem{theorem}{Theorem}[section]
 \newtheorem{lemma}[theorem]{Lemma}
 
 \newtheorem{corollary}[theorem]{Corollary}
 
 \newtheorem*{theorem*}{Theorem}

\theoremstyle{definition}
 \newtheorem{definition}[theorem]{Definition}
 
 \theoremstyle{remark}
 \newtheorem{remark}[theorem]{Remark}

 \newtheorem*{acknowledgments}{Acknowledgements}

\newcommand{\N}{\mathbf{N}}
\newcommand{\Q}{\mathbf{Q}}

\newcommand{\Sch}{\mathsf{Sch}}
\newcommand{\Var}{\mathsf{Var}}

\newcommand{\kr}{\,\begin{picture}(-1,1)(-1,-2)\circle*{2}\end{picture}\ }

\newcommand{\Spec}{\operatorname{Spec}}

\newcommand{\holim}{\mathrm{holim}}
\newcommand{\fil}{\mathrm{Fil}}
\newcommand{\mf}{\operatorname{MF}}

\newcommand{\dr}{\mathrm{dR}}
\newcommand{\drw}{\mathrm{dRW}}
\newcommand{\h}{\mathrm{h}}
\newcommand{\hk}{\mathrm{HK}}
\newcommand{\cris}{\mathrm{cris}}
\newcommand{\et}{\textrm{\'et}}
\newcommand{\nr}{\mathrm{nr}}
\newcommand{\ad}{\mathrm{ad}}

\newcommand{\OO}{\mathcal{O}}
\newcommand{\PP}{\mathcal{P}}
\newcommand{\A}{\mathcal{A}}

\renewcommand\labelenumi{(\roman{enumi})}
\renewcommand\theenumi\labelenumi

\begin{document}
\selectlanguage{british}

\title{A new proof of a vanishing result due to Berthelot, Esnault, and R\"ulling}
\author{ERTL Veronika }
\address{Keio University, Department of Mathematics, 3-14-1 Kouhoku-ku, Hiyoshi, Yokohama, 223-8522 Japan}
\address{Universit\"at Regensburg, Fakult\"at f\"ur Mathematik, Universit\"atsstra\ss{}e 31, 93053 Regensburg, Germany}
\email{ertl.vroni@gmail.com}  
\date{\today}
\thanks{The author was partly  supported by the DFG grant SFB 1085 ``Higher Invariants'' and by the Alexander von Humboldt Foundation  and the Japan Society for the Promotion of Science as a JSPS International Research Fellow.}
\maketitle 

\begin{abstract} 
The goal of this small  note is to give a  more concise proof of a result due to Berthelot, Esnault, and R\"ulling in \cite{BerthelotEsnaultRuelling2012}. 
For a regular, proper, and flat scheme $X$ over a discrete valuation ring of mixed characteristic $(0,p)$, it relates the vanishing of the cohomology  of the structure sheaf of the generic fibre of $X$ with the vanishing of the Witt vector cohomology of its special fibre. 
We use as a critical ingredient results and constructions by Beilinson \cite{Beilinson2014}  and Nekov\'a\v{r}--Nizio\l{} \cite{NekovarNiziol2016} related to the $\h$-topos over a $p$-adic field.
\end{abstract}
\selectlanguage{french}
\begin{abstract}
Le but de cette br\`eve  note est de donner une d\'emonstration plus courte d'un r\'esultat de Berthelot, Esnault et R\"ulling dans \cite{BerthelotEsnaultRuelling2012}. 
Pour un sch\'ema r\'egulier, propre et plat $X$ sur un anneau de valuation discr\`ete de caract\'eristique $(0,p)$, il lie la disparition de la cohomologie du faisceau structural de la fibre g\'en\'erique de $X$ \`a la disparition de la cohomologie de Witt de sa fibre sp\'eciale. 
On utilise de mani\`ere critique  des r\'esultats et des constructions de Beilinson \cite{Beilinson2014}  et Nekov\'a\v{r}--Nizio\l{} \cite{NekovarNiziol2016} concernant le $\h$-topos sur un corps $p$-adique.\\

\end{abstract}
\medskip

\selectlanguage{british}

\section*{}

According to \cite{Beilinson2012} schemes of semistable reduction over a complete discrete valuation ring $\OO_K$ with perfect residue field form a basis of the $\h$-topology on the category $\Var(K)$ of varieties over the fraction field $K$ of $\OO_K$. 
As a consequence, $\h$-sheafification makes it sometimes possible to generalise constructions or results from schemes of semistable reduction to varieties over a $p$-adic field.

In this small note we want to illustrate the advantages of this technique and give a shorter and, as we hope, more conceptual proof of the following vanishing result due to Berthelot, Esnault, and R\"ulling in \cite[Thm.~1.3]{BerthelotEsnaultRuelling2012}.

\begin{theorem*}[P.~Berthelot, H.~Esnault, K.~R\"ulling]
Let $R$ be a discrete valuation ring  of mixed characteristic with fraction field $K$ and perfect residue field $k$, and let $X$ be a regular proper flat scheme over $R$. 
Assume that $H^q(X_K,\OO)=0$ for some $q\geqslant 0$.
Then $H^q(X_0,W\OO)_{\Q}=0$ as well.
\end{theorem*}

As a consequence of this result the authors obtain, under the additional assumption that  $k$ is finite, a congruence on the number of rational points of $X$ with values in finite extensions of $k$. 
As explained in \cite{BerthelotEsnaultRuelling2012} this fits into the general analogy between the vanishing of Hodge numbers for varieties over a field of characteristic $0$ and congruences on the number of rational points with values in finite extensions for varieties over a finite field.

 The above theorem itself is an application of $p$-adic Hodge theory. In \cite[Thm.~2.1]{BerthelotEsnaultRuelling2012} the semistable case is discussed which we recall here briefly as it provides a guideline for our proof.
 
 Thus in the situation of the theorem, let $X/R$  be of semistable reduction. Without loss of generality one can assume that $R$ is a complete discrete valuation ring. 
 Endow $R$ and $X$ with  the canonical log structure denoted by $R^\times$ and $X^\times$,  the special fibre $X_0$ with the pull-back log structure denoted by $X_0^\times$, and the Witt vectors $W(k)$ with the log-structure associated to $(1\mapsto 0)$ denoted by $W(k)^0$. 
 Consider the log crystalline cohomology groups $H^q_{\cris}(X_0^\times/W(k)^0))$. 
 This cohomology is sometimes called the Hyodo--Kato cohomology. 
 In the case at hand, they can  be computed by the logarithmic de~Rham--Witt complex $W\omega^{\kr}$ which leads to a spectral sequence
 	$$
 	E_1^{ij}=H^j(X_0^\times,W\omega^i_{\Q})\Rightarrow H^{i+j}_{\cris}(X_0^\times/W(k)^0)_{\Q}
 	$$
endowed with a Frobenius action. 
On the left this Frobenius action is induced by $p^iF$, where $F$ is the Witt vector Frobenius, whereas on the right the Frobenius action $\varphi$ is induced by the absolute Frobenii of $X_0$ and $W(k)$. 
Similarly to the classical case, it follows that this spectral sequence degenerates at $E_1$ and that $H^j(X_0^\times,W\omega^i_{\Q})$ corresponds to the part of $H^{i+j}_{\cris}(X_0^\times/W(k)^0)_{\Q}$ where Frobenius has slope in $[i,i+1[$. 
Hence one obtains a canonical quasi-isomorphism
	$$
	H^q_{\cris}(X_0^\times/W(k)^0)_{\Q}^{< 1}\xrightarrow{\sim} H^q(X_0,W\OO)_{\Q}.
	$$
In other words, this means that Witt vector cohomology, which we want to study, corresponds to the part of crystalline cohomology of Frobenius slope $< 1$.
The cohomology group $H^q_{\cris}(X_0^\times/W(k)^0)_{\Q}$ is also equipped with a monodromy operator $N$ and a Hyodo--Kato isomorphism
	$$
	\iota_{\dr,\pi}: H^q_{\cris}(X_0^\times/W(k)^0)\otimes_{W(k)}K\xrightarrow{\sim} H^q_{\dr}(X_K),
	$$
where the de~Rham cohomology on the right hand side is equipped with the Hodge filtration. In particular, $H^q_{\cris}(X_0^\times/W(k)^0)_{\Q}$ can be regarded as an admissible filtered $(\varphi,N)$-module which implies that its Newton polygon lies above its Hodge polygon. 
But by assumption $H^q(X_K,\OO)=0$, which means that the part where the Hodge  slope is $<1$ vanishes. 
Hence  the same is true for the part where the Newton slope is $<1$, which is as we have seen isomorphic to $H^q(X_0,W\OO)_{\Q}$. 
This concludes the proof.

The philosophy behind this proof is that the cohomology groups  $H^q(X_K,\OO)$ and $H^q(X_0,W\OO)_{\Q}$, which are mathematical invariants associated to the generic and the special fibre, respectively, are in a certain sense part of a more comprehensive theory, namely absolute $p$-adic Hodge cohomology, which is realised in the category of admissible filtered $(\varphi,N)$-modules. The inherent structure provides intricate relations between invariants of the special and generic fibre. 

We realised that it is possible to use a very similar argument to obtain the more general statement of the theorem. 
In fact, it allows us to prove the theorem in a slightly more general case, namely for a proper, reduced and flat scheme over a discrete valuation ring $R$ of mixed characteristic, such that the generic fibre has at most Du Bois singularities. 

Let  $X$ be such a scheme.
Again, we want to interpret the cohomology groups $H^q(X_K,\OO)$ and $H^q(X_0,W\OO)_{\Q}$ in terms of  absolute $p$-adic Hodge cohomology. 
The ``right'' realisation category in this more general case is the category of admissible filtered $(\varphi,N,G_K)$-modules. 
In fact, D\'eglise and Nizio\l{} identified absolute $p$-adic Hodge cohomology for $K$-varieties (even without integral model) in \cite{DegliseNiziol2018}, 
where they look at the $\h$-sheafification on $K$-varieties  of a certain presentation of Hyodo--Kato  cohomology introduced by Beilinson.
According to their results, the associated cohomology groups $H^q_{\hk,\h}(X_K)$ of  the generic fibre $X_K$ can be seen as admissible filtered $(\varphi,N,G_K)$-module. 
For these the Newton and Hodge polygons are related in the same way as for $(\varphi,N)$-modules, i.e. their Newton polygon   lies above their Hodge polygon. 

We need now descent for the $\h$-topology to finish the proof. 
On the side of de~Rham cohomology a descent result due to Huber and J\"order in \cite{HuberJoerder2014}  shows that in the case of Du Bois singularities the part where the Hodge slope is $<1$ is exactly $H^q(X_K,\OO)$, which vanishes by assumption. 
This means that the part where the Newton slope of $H^q_{\hk,\h}(X_K)$ is $<1$ vanishes as well. 
We use a descent result for Witt vector cohomology due to Bhatt and Scholze in \cite{BhattScholze2017} to show that this is in fact $H^q(X_0,W\OO)_{\Q}$, which allows us to conclude.

Note that in our proof, as well as in the original one, descent results play an important role. 
In \cite{BerthelotEsnaultRuelling2012}, the authors work very explicitly with proper hypercovers over a discrete valuation ring to reduce to the semistable case. 
However, as they remark, the special fibre of such a hypercovering might not be a proper hypercover of the special fibre of the original scheme.
This is why they have to show an injectivity theorem  \cite[Thm.~1.5]{BerthelotEsnaultRuelling2012} which turns out to be very subtle. 

Our proof relies on the very sophisticated work of D\'eglise--Nizio\l{} and Nekov\'a\v{r}--Nizio\l{} which allows us to abstractly identify $H^q(X_K,\OO)$ and $H^q(X_0,W\OO)_{\Q}$ as part of a bigger picture.
Consequently,  we don't have to reduce to the semistable case, but nevertheless we have to invoke Bhatt--Scholze's descent theorem to extract the desired information about the special fibre.

\subsection*{Notation and conventions}

All schemes considered are separated and of finite type over a base $S$. 
For a fixed base scheme $S$ we denote this category by $\Sch(S)$, and by $\Var(S)$ the category of separated reduced schemes of finite type. When $S = \Spec R$ is affine, we write $\Sch(R)$ for $\Sch(S)$ and $\Var(R)$ for $\Var(S)$.

The general set-up throughout this note is as follows. Let $\OO_K$ be a complete discrete valuation ring of mixed characteristic $(0,p)$, with fraction field $K$ and perfect residue field $k$. 
Furthermore, we assume the valuation on $K$ normalised so that  $v(p)=1$.  
As usual denote by $\overline{K}$ an algebraic closure of $K$ and by $\OO_{\overline{K}}$ the integral closure of $\OO_K$ in $\overline{K}$. Let $W(k)$ be the ring of Witt vectors of $k$,  $K_0$ the fraction field of $W(k)$, and $K_0^{\nr}$ its maximal unramified extension.  
Let $\sigma$ be the absolute Frobenius on $W(\overline{k})$. For a scheme $X/\OO_K$ denote by $X_n$, for $n\in\N$, its reduction modulo $p^n$ and let $X_0$ and $X_K$ be its special and generic fibre respectively.

By abuse of notation, we denote by $\OO_K$, $\OO_K^\times$, and $\OO_K^0$ the scheme $\Spec \OO_K$ regarded as log scheme with the trivial, canonical (i.e. associated to the closed point) and $(\N\rightarrow \OO_K, 1\mapsto 0)$-log structure respectively, and similarly for $W(k)$ and $k$.

\begin{acknowledgments}
I would like to thank Shane Kelly and Wies\l{}awa Nizio\l{} for stimulating discussions related to this paper, and H\'el\`ene Esnault for very helpful explanations on her paper \textit{Rational points over finite fields for regular models of algebraic varieties of Hodge type $\geqslant 1$} with Berthelot and R\"ulling.
\end{acknowledgments}

 \section{The logarithmic de~Rham--Witt complex}\label{sec:1}

 It is common  when one studies non-smooth objects to consider  log versions of the usual complexes appearing in the different cohomology theories.  
 In \cite[Sec.~2]{HyodoKato1994} Hyodo and Kato describe the log crystalline site for schemes with fine log structure as a generalisation of the usual crystalline site. 
 One case which is relatively well studied, is the semistable case in positive characteristic $p$, and the case of semistable reduction in mixed characteristic $(0,p)$. 
This is a special case of fine log schemes of Cartier type over $k^0$ or $\OO_K^\times$. 

The Hyodo--Kato complex  computes   log crystalline cohomology over $W(k)^0$. 
There are different quasi-isomorphic presentations of this complex. One that is particularly useful to us is the de~Rham--Witt presentation233 from \cite[Sec.~4]{HyodoKato1994} (see also \cite[Sec.~1]{Lorenzon2002}).

For  a fine log scheme $Y$ of Cartier type over $k^0$, let $(Y/W_n(k)^0)_{\cris}$ be its log crystalline site over $W_n(k)$ endowed with the $(1\mapsto 0)$-log structure, let $\OO_{Y/W_n(k)^0}$ be its crystalline structure sheaf, and let $u_{Y/W_n(k)^0}:(Y/W_n(k)^0)_{\cris}\rightarrow Y_{\et}$ be the canonical morphism of sites. 

\begin{definition}
The logarithmic Witt differentials  of $Y$ of degree $i$ and level $n\geqslant 1$ are defined as
	$$
	W_n\omega^i_Y:= R^i u_{Y/W_n(k)^0,\ast}\OO_{Y/W_n(k)^0}.
	$$
\end{definition}
According to \cite[Cor.~1.17]{Lorenzon2002} $W_n\omega^i_Y$ is a coherent $W_n\OO_Y$-module. 
For $n$ fixed $W_n\omega^{\kr}_Y$ is a complex called the logarithmic de~Rham--Witt complex of level $n$.
By  \cite[Prop.~(4.6)]{HyodoKato1994} there is a canonical isomorphism of $W_n(k)$-algebras 
	$$
	W_n\omega^0_Y\cong W_n\OO_Y.
	$$
The Frobenius $F$, Verschiebung $V$ and projection operators extend from the Witt vectors to the de~Rham--Witt complex, and are subject to certain relations. 
In particular,  $W_n\omega^{\kr}_Y$ is a differentially graded $W_n\OO_Y$-algebra, which computes log crystalline cohomology, i.e. there is by \cite[Thm.~(4.19)]{HyodoKato1994}  a canonical quasi- isomorphism
	\begin{equation}\label{equ:dRW-cris}
	W_n\omega^{\kr}_Y\xrightarrow{\sim} Ru_{Y/W_n(k)^0,\ast}\OO_{Y/W_n(k)^0}
	\end{equation}
compatible with Frobenius and the canonical projections as $n$ varies. 
Here the Frobenius on the left hand side is given by $p^i F$ in degree $i$, while on the right hand side, it is induced by the absolute Frobenius of $Y$ and the endomorphism $\sigma$ of $W_n(k)$. 
The compatibility of this  morphism with the canonical projections was treated properly in \cite[\S~8]{Nakkajima2005}.

Finally, for $n$ varying we obtain a projective system   $\{W_n\omega^{\kr}_Y\}_{n\geqslant 1}$  differentially graded algebras.

\begin{definition}
We call the limit
	$$
	W\omega^{\kr}_Y:=\varprojlim W_n\omega^{\kr}_Y
	$$
the logarithmic de~Rham--Witt complex of $Y$ over $W(k)^0$. 
It is equipped with operators called Frobenius and Verschiebung induced by $F$ and $V$ which satisfy the usual relations. 
\end{definition}

The log crystalline cohomology of $Y$ over $W(k)^0$ is
	$$
	R\Gamma_{\cris}(Y/W(k)^0):=\holim R\Gamma_{\et}(Y, Ru_{Y/W_n(k)^0,\ast}\OO_{Y/W_n(k)^0}).
	$$
and we have by \cite[Cor.~1.23]{Lorenzon2002} the following statement. 

\begin{lemma}[P.~Lorenzon]
If $Y$ is a proper fine log scheme  of Cartier type over $k^0$, then  $R\Gamma_{\et}(Y,W\omega^{\kr}_Y)$ is a perfect complex and the canonical morphism (\ref{equ:dRW-cris}) induces a natural isomorphism 
	$$
	\lambda: R\Gamma_{\et}(Y,W\omega^{\kr})\xrightarrow{\sim}R\Gamma_{\cris}(Y/W(k)^0).
	$$
\end{lemma}

In  case $Y=X_0$ is the special fibre of a fine log scheme $X$ of Cartier type over $\OO_K^\times$, we write $R\Gamma_{\hk}(X):=R\Gamma_{\cris}(Y/W(k)^0)$ and call it the Hyodo--Kato complex of $X$. 
In this case, there is yet another presentation of the rational Hyodo--Kato complex due to Beilinson \cite[1.16.1]{Beilinson2014} denoted by $R\Gamma_{\hk}^B(X)$.
Without going into details we remark 
that there is a natural quasi-isomorphism \cite[(28)]{NekovarNiziol2016} 
	\begin{equation}\label{equ:BHK-HK}
	\kappa: R\Gamma_{\hk}^B(X)\xrightarrow{\sim} R\Gamma_{\hk}(X)_{\Q}
	\end{equation}
compatible with Frobenius.
The advantage of Beilinson's definition is that it admits a \textit{nilpotent} monodromy $N$ as explained in \cite[\S~3.1]{NekovarNiziol2016}. We will see later why this is of interest for us.

\section{Hyodo--Kato complexes for $K$-varieties}

 In \cite{NekovarNiziol2016} Nekov\'a\v{r} and Nizio\l{} describe how to extend several $p$-adic cohomology theories for log schemes over $k$ or $\OO_K$ to $K$-varieties. 
 This technique is based on observations due to Beilinson in  \cite{Beilinson2012} and we start by recalling some of the relevant notions.

\begin{definition}
For a field $K$ of characteristic $0$ a  geometric pair is an open embedding $i:U\hookrightarrow \overline{U}$ of $K$-varieties such that $U$ is dense in $\overline{U}$ and $\overline{U}$ is proper. A geometric pair is called a   normal crossings pair,  if $\overline{U}$ is regular and $\overline{U}\backslash U$ is a divisor with normal crossings. It is said to be  strict, if the irreducible components of $\overline{U}\backslash U$ are regular. 
\end{definition}

This definition can be adjusted to the arithmetic setting as follows.

\begin{definition}\label{def:Pairs}
An arithmetic $K$-pair is an open embedding $i:U\hookrightarrow \overline{U}$ of a $K$-variety $U$ with dense image in a reduced proper flat $\OO_K$-scheme $\overline{U}$. 
Such a pair is called a   semistable pair  if $\overline{U}$ is regular, $\overline{U}\backslash U$ is a divisor with normal crossings, and the closed fibre $\overline{U}_0$ is reduced. It is said to be   strict, if the irreducible components of $\overline{U}\backslash U$ are regular. 
\end{definition}

As explained in \cite[2.3]{NekovarNiziol2016} one can regard a semistable pair $(U,\overline{U})$ as a log scheme $\overline{U}$ over $\OO_K^\times$ equipped with the log structure associated to the divisor $\overline{U}\backslash U$. 
As such it is a proper log smooth fine log scheme  of Cartier type over $\OO_K^\times$. 
In particular, its special fibre $\overline{U}_0$ equipped  with the pull back log structure is a  proper log smooth fine log scheme  of Cartier type over $k^0$. 

Following \cite{NekovarNiziol2016} we denote by $\PP^{ar}_K$, $\PP^{ss}_K$ and $\PP^{nc}_K$ respectively the category of arithmetic, semistable and normal crossings pairs over $K$ respectively. 
Moreover, we denote by $\PP_K^{log}$ the subcategory of $\PP_K^{ar}$ of log schemes $(U,\overline{U})$ which are   log smooth over $\OO_{\overline{U}}(\overline{U})^\times$.

A key point in Beilinson's work  is that the categories $\PP^{ar}_K$, $\PP^{ss}_K$,  $\PP_K^{log}$,  and $\PP^{nc}_K$ respectively form a base for the $\h$-site $(Var/K)_{\h}$ of $K$-varieties in the sense that there is an equivalence between the associated $h$-topoi \cite[2.5~Prop.]{Beilinson2012}. 
Keeping in mind that alterations are $\h$-morphisms, Beilinson uses the following formulation of  de~Jong's alteration theorem \cite[2.3~Thm.]{Beilinson2012}. 

\begin{theorem}[A.J. de~Jong]
Every geometric pair admits a strict normal crossings alteration. 
Every arithmetic pair over $K$ or $\overline{K}$ admits a strict semistable alteration. 
Alterations can be chosen in such a way that $\overline{V}$ is projective.
\end{theorem}

Nekov\'a\v{r} and Nizio\l{} explain how to $\h$-sheafify the rational Hyodo--Kato and the Beilinson--Hyodo--Kato complexes on $\Var(K)$ \cite[Sec. 3.3]{NekovarNiziol2016}.  
Denote the resulting sheaves by $\A_{\hk}$ and $\A_{\hk}^B$.  
The same procedure can be done to the logarithmic de~Rham--Witt complex $W\omega^{\kr}$ defined above. 

\begin{definition}
Let $\A_{\drw}$ be the $\h$-sheafification of the presheaf 
  	$$
  	(U,\overline{U})\mapsto R\Gamma_{\et}((U,\overline{U})_0,W\omega^{\kr})_{\Q}
  	$$
on the category $\PP_K^{log}$ of proper log smooth fine log schemes  of Cartier type over $\OO_K^\times$. 
This results in an $\h$-sheaf of commutative dg $K_0$-algebras on $\Var(K)$. 
\end{definition}

The canonical maps $\kappa:R\Gamma_{\hk}^B(U,\overline{U})\rightarrow R\Gamma_{\hk}(U,\overline{U})_{\Q}$ \cite[Sec.~3.3]{NekovarNiziol2016} and $\lambda:R\Gamma_{\et}((U,\overline{U})_0,W\omega^{\kr})_{\Q} \rightarrow R\Gamma_{\hk}(U,\overline{U})_{\Q}$ \cite[Thm.~(4.19)]{HyodoKato1994} $\h$-sheafify and we obtain  functorial quasi-isomorphisms of $\h$-sheaves
  	\begin{equation}\label{equ:quasi-iso1}
  	\A_{\drw} \cong \A_{\hk}\cong\A_{\hk}^B.
  	\end{equation}

\begin{lemma}\label{lem:hdescenttheorem}
For any proper log smooth fine log scheme of Cartier type  $(U,\overline{U})\in \PP_K^{log}$  over $\OO_K^\times$,   the canonical map
  	$$
  	R\Gamma_{\et}((U,\overline{U})_0,W\omega^{\kr}_{\Q})\rightarrow R\Gamma_{\h}(U,\A_{\drw})
  	$$
is a quasi-isomorphism.
\end{lemma}
\begin{proof}
According to \cite[Prop.~3.18]{NekovarNiziol2016}  analogous statements are true for the Hyodo--Kato and Beilinson--Hyodo--Kato complexes. 
Because of the canonical quasi-isomorphism (\ref{equ:quasi-iso1}) the statement follows.
\end{proof}

 \section{Admissible filtered $(\varphi, N, G_K)$-modules}
 
 The Beilinson--Hyodo--Kato complex has additional structure and as such fits into the theory of $p$-adic Galois representations \cite{Fontaine1994}. 
 Indeed, it's cohomology groups are admissible filtered  $(\varphi,N,G_K)$-modules. 
 Let us explain this. 
 
 \begin{definition}
	\begin{enumerate}
\item A filtered $\varphi$-module is a triple $(M_0,\varphi,\fil^{\kr})$, where $M_0$ is a finite dimensional $K_0$-vector space, with a $\sigma$-semi-linear isomorphism $\varphi: M_0\rightarrow M_0$, called the Frobenius map and a decreasing, separated, exhaustive filtration $\fil^{\kr}$ on $M=M_0\otimes_{K_0}K$ called the Hodge filtration. 
\item A filtered $(\varphi,N)$-module $(M_0,\varphi,N,\fil^{\kr})$ consists of a filtered $\varphi$-module $(M_0,\varphi,\fil^{\kr})$ together with a $K_0$-linear monodromy operator $N$ on $M_0$ satisfying the relation $N\varphi=p\varphi N$. 
\item A filtered $(\varphi,N,G_K)$-module is a tuple $(M_0,\varphi,N,\rho,\fil^{\kr})$ where $M_0$ is a finite dimensional $K_0^{\nr}$-vector space, $\varphi: M_0\rightarrow M_0$ is a Frobenius map, $N: M_0\rightarrow M_0$ is a $K_0^{\nr}$-linear monodromy operator such that $N\varphi=p\varphi N$, $\rho$ is a $K_0^{\nr}$-semi-linear action of $G_K$ on $M_0$ factoring through a quotient of the inertia group and commuting with $\varphi$ and $N$, and $\fil^{\kr}$ is a decreasing, separated, exhaustive filtration of $M=(M_0\otimes_{K_0^{\nr}}\overline{K})^{G_K}$.
	\end{enumerate}
\end{definition}

Denote by 
	$$
	\mf_K^{\ad}(\varphi) \subset \mf_K^{\ad}(\varphi,N) \subset \mf_K^{\ad}(\varphi,N,G_K)
	$$ 
the categories of  admissible filtered $\varphi$-modules, $(\varphi,N)$-modules and $(\varphi,N,G_K)$-modules, where admissible is meant in the sense of \cite{Fontaine1994}. 
They are equivalent to crystalline, semistable, and potentially semistable Galois representations \cite{ColmezFontaine2000,Berger2002}. 
The categories of admissible filtered $\varphi$-, $(\varphi,N)$, and $(\varphi,N,G_K)$-modules are known to be Tannakian (c.f \cite[\S$\;$4.3.4]{Fontaine1994}). 
Thus it makes sense to consider their respective bounded derived dg categories denoted by $\mathcal{D}^{\flat}(\mf_K^{\ad}(\varphi))$, $\mathcal{D}^{\flat}(\mf_K^{\ad}(\varphi,N))$ and $\mathcal{D}^{\flat}(\mf_K^{\ad}(\varphi,N,G_K))$ respectively. 

A filtered $\varphi$-, $(\varphi,N)$-, or $(\varphi,N,G_K)$-module $M_0$ has both, a Newton polygon, associated to the eigenvalues of the Frobenius morphism $\varphi$ on $M_0$, and a Hodge polygon, associated to the filtration $\fil^{\kr}$ on $M$. 
If the Newton polygon lies above the Hodge polygon, $M_0$ is weakly admissible (c.f. \cite[4.4.6~Rem.]{Fontaine1994}). 

It turns out that this is the critical piece of information that will allow us to relate $H^q(X_K,\OO_{X_K})$ and $H^q(X_0,W\OO)_{\Q}=0$. 
This is possible because the Beilinson--Hyodo--Kato complex $R\Gamma_{\h}(Z_{\overline{K}},\A_{\hk}^B)$ of a $K$-variety is an object in $\mathcal{D}^{\flat}(\mf_K^{\ad}(\varphi,N,G_K))$. 
By definition it is equipped with a Frobenius $\varphi$, and a nilpotent monodromy operator $N$. 
Note that this is a difference from the usual Hyodo--Kato complex, where the monodromy is at best homotopically nilpotent. 
This is  crucial  to $\h$-sheafify the Hyodo--Kato morphism,  relating the (Beilinson--)Hyodo--Kato complex to the de~Rham complex, which provides the filtration. 

For a proper log smooth fine log scheme of Cartier type  $(U,\overline{U})\in \PP_K^{log}$  over $\OO_K^\times$ there is a morphism (c.f. \cite[(22)]{NekovarNiziol2016}, \cite[Sec.~5]{HyodoKato1994})
	$$
	\iota_{\dr,\pi}:R\Gamma_{\hk}(U,\overline{U})_{\Q} \rightarrow R\Gamma_{\dr}(U,\overline{U}_K)
	$$
called the Hyodo--Kato morphism which becomes a $K$-linear functorial quasi-isomorphism after tensoring with $K$. 
However, it depends on the choice of a uniformiser $\pi$ of $\OO_K$ and is therefore not suitable for $\h$-sheafification. 
By contrast, there is a morphism \cite[1.16.3]{Beilinson2014}
	$$
	\iota_{\dr}^B:R\Gamma_{\hk}^B(U,\overline{U}) \rightarrow R\Gamma_{\dr}(U,\overline{U}_K)
	$$
independent of the choice of a uniformiser which is also a $K$-linear functorial quasi-isomorphism after tensoring with $K$. 
For more details see \cite[Ex.~3.5(1)]{NekovarNiziol2016}. 
The two morphisms are compatible with the comparison map (\ref{equ:BHK-HK})
	$$
	\xymatrix{R\Gamma_{\hk}(U,\overline{U})_{\Q} \ar[r]^{\iota_{\dr,\pi}} & R\Gamma_{\dr}(U,\overline{U}_K)\\
	R\Gamma_{\hk}^B(U,\overline{U}) \ar[u]^\kappa_{\sim} \ar[ur]_{\iota_{\dr}^B} & }
	$$
as explained in \cite{NekovarNiziol2016} after Ex.~3.5. 

The map $\iota_{\dr}^B$ can be $\h$-sheafified. 
For this we also have to $\h$-sheafify the de~Rham complex on $\Var(K)$. 
Thus,  consider  the presheaf
	$$
	(U,\overline{U})\mapsto R\Gamma((U,\overline{U}),\Omega^{\kr})
	$$
of filtered $K$-algebras on $\PP_K^{nc}$ and let $\A_{\dr}$ be its $\h$-sheafification. 
This results in  an $\h$-sheaf of commutative filtered dg $K$-algebras on $\Var(K)$. 
It can be identified with Deligne's de~Rham complex equipped with Deligne's Hodge filtration $\fil^{\kr}$ (c.f. \cite[Prop.~7.4 and Thm.~7.7]{HuberJoerder2014}). 
Moreover, Beilinson showed in \cite[2.4]{Beilinson2012}  that for $(U,\overline{U})\in\PP_K^{nc}$ the canonical map
	$$
	R\Gamma_{\dr}(U,\overline{U})\xrightarrow{\sim} R\Gamma_{\h}(U,\A_{\dr})
	$$
is a quasi-isomorphism.

There are analogous statements over $\overline{K}$, and for $Z\in\Var(K)$ the projection $\varepsilon: Z_{\overline{K},\h}\rightarrow Z_{\h}$ of sites induces pull-back maps
 	\begin{eqnarray}\label{epsilon*}
 	&\varepsilon^\ast: R\Gamma_{\h}(Z,\A_{\hk}^B) \rightarrow R\Gamma_{\h}(Z_{\overline{K}},\A_{\hk}^B)^{G_K}&\\
 	&\varepsilon^\ast: R\Gamma_{\h}(Z,\A_{\dr}) \rightarrow R\Gamma_{\h}(Z_{\overline{K}},\A_{\dr})^{G_K}&\nonumber\\
 	\end{eqnarray}
where by \cite[Prop.~3.22]{NekovarNiziol2016} the first one is a quasi-isomorphism and the second one is a filtered quasi-isomorphism.  
The Beilinson--Hyodo--Kato map extends to $R\Gamma_{\h}(Z_{\overline{K}},\A^B_{\hk})\rightarrow R\Gamma_{\h}(Z_{\overline{K}},\A_{\dr})$,
which induces a quasi-isomorphism
	$$
	R\Gamma_{\h}(Z_{\overline{K}},\A^B_{\hk})\otimes_{K_0^{\nr}} \overline{K} \rightarrow R\Gamma_{\h}(Z_{\overline{K}},\A_{\dr}).
	$$
This identification provides the last piece of data necessary and by \cite[2.20]{DegliseNiziol2018} we have the following statement.

\begin{lemma}[F.~D\'eglise, W.~Nizio\l]
Let $Z\in \Var(K)$. Then $R\Gamma_{\h}(Z_{\overline{K}},\A_{\hk}^B)$ with the Frobenius $\varphi$, the monodromy operator $N$, the canonical $G_K$-action, and the Hodge filtration on $R\Gamma_{\h}(Z,\A_{\dr})$ is an object in $\mathcal{D}^{\flat}(\mf_K^{\ad}(\varphi,N,G_K))$.
\end{lemma}

As a consequence we obtain the promised statement which relates the Hodge and the Newton polygon of  $H^q_{\h}(Z_{\overline{K}},\A_{\hk}^B)$.

\begin{lemma}\label{lem:polygon}
Let $Z\in\Var(K)$. For any  $q\in \N_0$, the Newton polygon of $H^q_{\h}(Z_{\overline{K}},\A_{\hk}^B)$ lies above its Hodge polygon. 
\end{lemma}
\begin{proof}
By the previous lemma $H^q_{\h}(Z_{\overline{K}},\A_{\hk}^B)$ is an admissible filtered $(\varphi,N,G_K)$-module. 
Therefore it is weakly admissible in the sense of Fontaine \cite[5.6.7~Thm.~(vi)]{Fontaine1994}. 
This means as remarked in \cite[4.4.6~Rem.]{Fontaine1994} that for each $q$ the Newton polygon of $H^q_{\h}(Z_{\overline{K}},\A_{\hk}^B)$  lies above its Hodge polygon.
\end{proof}

 \section{Two spectral sequences}

In this section we consider two spectral sequences, one for de~Rham cohomology and one for Hyodo--Kato cohomology, which are very similar in spirit for they are related to the Hodge and Newton slope of a $(\varphi,N,G_K)$-module in a geometric situation. 
The first one is an $\h$-sheafified version of the Hodge-to-de~Rham spectral sequence. 

For this we introduce the $\h$-sheafifications $\A_{\dr}^i$ of the differential sheaves $\Omega^i$, $i\geqslant0$, on $\Var(K)$. 
To be consistent with the constructions in the previous sections, we can think of $\A_{\dr}^i$ as the $\h$-sheafification of the presheaf
	$$
  	(U,\overline{U})\mapsto \Gamma((U,\overline{U}), \Omega^i)
  	$$
on the category $\PP_K^{nc}$ of normal crossing pairs. 
It gives a coherent $\h$-sheaf on $\Var(K)$. 

Now the Hodge filtration of $\A_{\dr}$ induces a spectral sequence for which Huber and J\"order in \cite[Thm.~7.7]{HuberJoerder2014} prove the following.

\begin{lemma}[A.~Huber, C.~J\"order]\label{lem:Hodge-deRham}
Let $Z\in\Var(K)$ be proper. 
Then the Hodge-to-de~Rham spectral sequence
	$$
  	E_1^{rs}=H^s_{\h}(Z,\A_{\dr}^r) \Rightarrow H^{r+s}_{\h}(Z,\A_{\dr})
  	$$
degenerates at $E_1$.
\end{lemma}
It becomes immediately clear that the Hodge filtration on $\A_{\dr}$ induces Deligne's Hodge filtration as mentioned above. 

\begin{corollary}\label{cor:Hodge-0-Part}
Let $Z\in\Var(K)$. 
Then for any $q\geqslant0$ the Hodge-to-de~Rham spectral sequence yields an isomorphism
  	$$
  	H^q_{\h}(Z,\A_{\dr})^{<1}\xrightarrow{\sim} H^q_{\h}(Z,\A_{\dr}^0),
  	$$
where on the left hand side we mean the part of $H^q_{\h}(Z,\A_{\dr})$ with Hodge slope $<1$. 
\end{corollary}

We come now to an analogous statement for the Hyodo--Kato cohomology. 
Thus for $i\geqslant 0$ consider  the $\h$-sheafification $\A_{\drw}^i$ of the presheaf
  	$$
  	(U,\overline{U})\mapsto \Gamma((U,\overline{U})_0, W\omega^i)_{\Q}
  	$$
on the category $\PP_K^{log}$ of proper log smooth fine log scheme  of Cartier type over $\OO_K$. 
This is a quasi-coherent $\h$-sheaf on $\Var(K)$. 

\begin{lemma}\label{Lem:SlopeSpectralSequence}
Let $Z\in\Var(K)$. There is a spectral sequence 
  	$$
  	E_1^{rs}=H^s_{\h}(Z,\A_{\drw}^r) \Rightarrow H^{r+s}_{\h}(Z,\A_{\drw})
  	$$
which is Frobenius equivariant and degenerates at $E_1$.
\end{lemma}
\begin{proof}
The existence of the spectral sequence follows as in the classical case, meaning it is induced from the naive filtration of the complex $\A_{\drw}$. 
As mentioned earlier, the Frobenius $F$ induces an endomorphism of the de Rham--Witt complex, which in each degree $r$ is given by $p^rF$. 
It $\h$-sheafifies well. 
Thus we can use the same argumentation as in \cite[III.3.1]{Illusie1979} to see that the spectral sequence is Frobenius equivariant. Namely, the Frobenius endomorphism of $\A_{\drw}$ fixes  $\A_{\drw}^{\geqslant r}$ for all $r\geqslant 0$ and thus it induces an endomorphism of the spectral sequence. 
On the abutment $H^{r+s}_{\h}(Z,\A_{\drw})$, it coincides with the Frobenius action $\varphi$. 
On $H^s_{\h}(Z,\A_{\drw}^r)$ it is given by $p^r F$.

To see that the spectral sequence degenerates, we will show that $H^s_{\h}(Z,\A_{\drw}^r)$ is a finite rank $K_0$-vector space for all $r,s\geqslant 0$.
More precisely, we will show that $(H^s_{\h}(Z,\A_{\drw}^r), p^r F)$ is canonically isomorphic to the part of $(H^{r+s}_{\h}(Z,\A_{\drw}),\varphi)$ which has slopes in $[r,r+1[$,   denoted  by  $(H^{r+s}_{\h}(Z,\A_{\drw}),\varphi)^{[r,r+1[}$. 
The statement then follows from the fact that the cohomology groups $H^{r+s}_{\h}(Z,\A_{\drw})$ are finite rank $K_0$-vector spaces \cite[p.$\;$5]{NekovarNiziol2016}.

Let $(U_{\kr},\overline{U}_{\kr})\rightarrow Z$ by an $\h$-hypercover of $Z$ by semistable pairs over $K$ which exists by de Jong's alteration theorem. 
For every $n$, let $K_{U_n}:=\Gamma(\overline{U}_{n,K},\OO_{\overline{U}_n})$ which is a finite product of finite extensions of $K$ labelled by the connected components of $\overline{U}_n$, that is $K_{U_n}=\prod K_{n,i}$. 
As we may assume that all the fields $K_{n,i}$ are Galois over $K$, we choose a finite Galois extension $(\OO_{K'},K')/(\OO_K,K)$ with residue field $k'$ such that $K'$ is Galois over all the fields $K_{n,i}$ for fixed $n$ (c.f.$\;$proof of \cite[Prop.$\;$3.20]{NekovarNiziol2016}). 
Then  $(U'_n,\overline{U}'_n):=  (U_n,\overline{U}_n)\times_{\OO_K} \OO_{K'}$ is a proper log smooth fine log scheme of Cartier type over $\OO_{K'}^\times$.
Base change for crystalline cohomology implies that   one has  isomorphisms
$$H^s((U'_n,\overline{U}'_n)_0, W\omega^r_{\Q}) \cong H^s((U_n,\overline{U}_n)_0, W\omega^r_{\Q}) \otimes_{K_0} K'_0$$
where $K'_0$ denotes the fraction field of $W(k')$. 
As the left hand side is a finite  dimensional $K'_0$-vector space, it follows that $H^s((U_n,\overline{U}_n)_0, W\omega^r_{\Q})$ is a finite dimensional $K_0$-vector space. We can do this for all $n$ separately.
It follows, that the spectral sequence
	$$
	E_1^{rs}=H^s((U_n,\overline{U}_n)_0, W\omega^r_{\Q})\Rightarrow H^{r+s}((U_n,\overline{U}_n)_0,W\omega^{\kr}_{\Q}),
	$$
which is again induced by the naive filtration of the logarithmic de Rham--Witt complex, degenerates at $E_1$, and that we have canonical  isomorphisms
$$\left(H^s((U_n,\overline{U}_n)_0,W\omega^r_{\Q}), p^r F\right) \cong \left(H^{r+s}((U_n,\overline{U}_n)_0,W\omega^{\kr}_{\Q}), \varphi\right)^{[r,r+1[}$$
for all $n$. 
Accordingly, there is a canonical  isomorphism
$$\left(H^s((U_{\kr},\overline{U}_{\kr})_0,W\omega^r_{\Q}), p^r F\right) \cong \left(H^{s+q}((U_{\kr},\overline{U}_{\kr})_0,W\omega^{\kr}_{\Q}), \varphi\right)^{[r,r+1[}.$$
of finite dimensional $K_0$-vector spaces. 

But as mentioned before,  $\varphi$ and $F$ sheafify well with respect to the $\h$-topology, so that we may take the limit over all $\h$-hypercovers of $Z$ by semistable pairs over $K$ and obtain a canonical isomorphism
	$$
	(H^s_{\h}(Z,\A_{\drw}^r),p^r F)\cong (H^{r+s}_{\h}(Z,\A_{\drw}),\varphi)^{[r,r+1[}
	$$
where the right hand side is the part of $(H^{r+s}_{\h}(Z,\A_{\drw}),\varphi)$ with slope in the interval $[r,r+1[$. 
In particular, the $H^s(Z,\A_{\drw}^r)$ $s\geqslant 0$,  are finite rank $K_0$-vector spaces and the spectral sequence from the statement degenerates.
\end{proof}

For obvious reasons this spectral sequence is called the slope spectral sequence. 
 
\begin{corollary}\label{cor:slope-0-Part}
Let $Z\in\Var(K)$. 
For any $q\geqslant0$ the slope spectral sequence yields an isomorphism
  	$$
  	H^q_{\h}(Z,\A_{\drw})^{<1}\xrightarrow{\sim} H^q_{\h}(Z,\A_{\drw}^0),
  	$$
where on the left hand side we mean the part of $H^q_{\h}(Z,\A_{\drw})$ where Frobenius acts with slope $<1$. 
\end{corollary}

This allows us to transfer the descent statement from Lemma~\ref{lem:hdescenttheorem} to differentials of degree zero.

\begin{corollary}\label{cor:deg0-hdescent}
For any proper log smooth fine log scheme of Cartier type $(U,\overline{U})\in\PP^{log}_K$ over $\OO_K^\times$, the canonical map 
	$$
  	R\Gamma_{\et}((U,\overline{U})_0,W\omega^0)_{\Q} \rightarrow R\Gamma_{\h}(U,\A_{\drw}^0)
  	$$
is a quasi-isomorphism.
\end{corollary}
\begin{proof}
For any $q\geqslant 0$ we have a commutative diagram
	$$
	\xymatrix{
	H^q((U,\overline{U})_0,W\omega^{\kr})_{\Q}^{<1} \ar[r]^{\sim} \ar[d]^{\sim} & H^q((U,\overline{U})_0,W\omega^0)_{\Q} \ar[d]\\
	H^q_{\h}(U,\A_{\drw})^{<1} \ar[r]^{\sim} & H^q_{\h}(U,\A_{\drw}^0)}
	$$
where the horizontal isomorphisms are induced from the classical and the $\h$-sheafified slope spectral sequence respectively and the vertical maps are the canonical morphisms. 
Since the left vertical map is an isomorphism by Lemma~\ref{lem:hdescenttheorem}, the right one is as well.
\end{proof}

We can take this a bit further. 

\begin{lemma}\label{lem:K-to-k}
Let $X$ be a reduced, proper and flat $\OO_K$-scheme of finite type. Then there is a quasi-isomorphism
$$R\Gamma_{\h}(X_K,\A_{\drw}^0)\cong R\Gamma_{\et}(X_0,W\OO)_{\Q}.$$
\end{lemma}
\begin{proof} 
Since $(X_K,X)$ forms an arithmetic pair in the sense of Beilinson (see Definition~\ref{def:Pairs}), it has by  de~Jong's   theorem a strictly semistable alteration $(U,\overline{U})\rightarrow (X_K,X)$. This gives rise to an $\h$-cover $U\rightarrow X_K$  of the generic fibre, and an $\h$-cover $\overline{U}_0\rightarrow X_0$ of the special fibre.
Denote by $(U_{\kr},\overline{U}_{\kr})\rightarrow (X_K,X)$ its \v{C}ech nerve. In particular $U_{\kr}\rightarrow X_K$ is an $\h$-hypercover of  the generic fibre of $X$ and $\overline{U}_{\kr,0}\rightarrow X_0$ is an $\h$-hypercover of its special fibre.

Using Corollary~\ref{cor:deg0-hdescent} we have
	$$R\Gamma_{\h}(X_K,\A_{\drw}^0) \xrightarrow{\sim} R\Gamma_{\h}(U_{\kr},\A_{\drw}^0) \xleftarrow{\sim} R\Gamma_{\et}((U_{\kr},\overline{U}_{\kr})_0,W\omega^0)_{\Q}  \cong  R\Gamma_{\et}(\overline{U}_{\kr,0},W\OO)_{\Q}.
	$$
However, by \cite[Prop.~11.41]{BhattScholze2017} rational Witt cohomology satisfies cohomological $\h$-descent and hence the right most expression is just $R\Gamma_{\et}(X_0,W\OO)_{\Q}$. 
\end{proof}

Therefore we can rewrite Corollary~\ref{cor:slope-0-Part}.

\begin{corollary}
Let $X$ be a reduced, proper and flat $\OO_K$-scheme of finite type. For any $q\geqslant0$ the slope spectral sequence yields an isomorphism
  	$$
  	H^q_{\h}(X_K,\A_{\drw})^{<1}\xrightarrow{\sim} H^q(X_0,W\OO)_{\Q}.
  	$$
\end{corollary}

\section{A vanishing theorem}
 
We can now put the pieces together to give a simplified proof of the vanishing theorem due to  Pierre Berthelot, H\'el\`ene Esnault and Kay R\"ulling.  
We will use the following statement about de~Rham cohomology groups. 
Note that $\h$-sheafifying the de~Rham complex results in Deligne's de~Rham complex \cite[Thm.~7.4]{HuberJoerder2014}, i.e. $H^q_{\h}(Z,\A_{\dr}) =H^q_{\dr}(Z)$. 

\begin{lemma}\label{lem:hodge-slope}
Let $Z\in\Var(K)$ be proper with only Du~Bois singularities, and assume that $H^q(Z,\OO)=0$ for some $q\geqslant 0$. 
Then the smallest Hodge slope of $H^q_{\dr}(Z)$  is at least $1$. 
\end{lemma}
\begin{proof}
If $Z$ has only Du~Bois singularities $H^q_{\h}(Z,\A_{\dr}^0)=H^q(Z,\OO)$ by \cite[Cor.~7.17]{HuberJoerder2014}.  As the $\h$-sheafified Hodge-to-de~Rham spectral sequence from Lemma~\ref{lem:Hodge-deRham} degenerates for a proper $K$-variety at $E_1$, the hypothesis  $H^q(Z,\OO)=0$ implies that the smallest Hodge slope of  $H^q_{\h}(Z,\A_{\dr})\cong H^q_{\dr}(Z)$ is at least $1$.
\end{proof}

\begin{remark}
In general we can say that for a proper $K$-variety $Z\in\Var(K)$ such that $H^q_{\h}(Z,\OO_{\h})=0$ for some $q$ the smallest Hodge slope of $H^q_{\dr}(Z)$  is at least $ 1$. 
\end{remark}

We obtain now the desired vanishing theorem in a slightly more general form than originally stated.
 
\begin{theorem}
 Let $X$ be a proper, reduced and flat scheme over  $\OO_K$, such that $X_K$ has at most Du~Bois singularities. Fix $q \in \N_0$. If $H^q(X_K,\OO)=0$, then  $H^q(X_0,W\OO)_{\Q}=0$.
\end{theorem}
\begin{proof}
Consider the cohomology group $H^q_{\h}(X_{\overline{K}},\A_{\hk}^B)$.
 By Lemma~\ref{lem:polygon} its Newton polygon lies above its Hodge polygon. 
 Because $X_K$ has only Du~Bois singularities Lemma~\ref{lem:hodge-slope} applies, which means that the smallest Hodge slope of $H^q_{\dr}(X_K)$ is $\geqslant 1$. 
 Therefore the part of $H^q_{\h}(X_{\overline{K}},\A_{\hk}^B)$, where the Newton slope  is $<1$ vanishes. 
 By definition this is exactly the part where Frobenius acts with slope $<1$.
Via the first quasi-isomorphism of (\ref{epsilon*}) and the fact  that the action of $G_K$ commutes with $\varphi$ we deduce the same for $H^q_{\h}(X_K,\A_{\hk}^B)$. 
Finally, the quasi-isomorphisms (\ref{equ:quasi-iso1}) imply $H^q_{\h}(X_K,\A_{\drw})^{<1}=0$ and hence, by Corollary~\ref{cor:slope-0-Part} 
	$$H^q_{\h}(X_{K}, \A_{\drw}^0)=0.$$
But according to Lemma~\ref{lem:K-to-k} this means that $H^q(X_0,W\OO)_{\Q}=0$ as well.
\end{proof}

\begin{remark}
As pointed out in \cite{BerthelotEsnaultRuelling2012} one easily generalises this result to a scheme $X$ as in the theorem, but over a discrete valuation ring $V$ which is not necessarily complete.
\end{remark}

\end{document}